\documentclass{amsart}
\usepackage{amsmath, amssymb, mathrsfs, verbatim, multirow}

\usepackage[all]{xy}
\usepackage{pifont}
\usepackage{float}
\usepackage{color}
\usepackage{enumitem}

\usepackage{tikz}
\usepackage{tikz-cd}
\usepackage{tkz-tab}
\usepackage{subcaption} 

\usepackage[hidelinks]{hyperref}

\newcommand{\Q}{\mathbb{Q}}

\newcommand{\F}{\mathbb{F}}
\newcommand{\Z}{\mathbb{Z}}
\newcommand{\N}{\mathbb{N}}
\newcommand{\R}{\mathbb{R}}

\newcommand{\lra}{\longrightarrow}

\newcommand{\Char}{{\rm char}}

\newcommand{\supp}{\mbox{\rm supp}}
\newcommand{\VR}{\mathcal{O}}
\newcommand{\MI}{\mathfrak{m}}

\newtheorem{theorem}{Theorem}[section]

\newtheorem{prop}[theorem]{Proposition}
\newtheorem{lemma}[theorem]{Lemma}
\newtheorem{corollary}[theorem]{Corollary}

\theoremstyle{definition}

\newtheorem{definition}[theorem]{Definition}
\newtheorem{example}[theorem]{Example}
\newtheorem{remark}[theorem]{Remark}

\begin{document}

\title{On defect in finite extensions of valued fields}
\author{Silva de Souza, C. H.$^1$ and Spivakovsky, M.$^2$}
\thanks{During the realization of this project, the first author was supported by a grant from Funda\c c\~ao de Amparo \`a Pesquisa do Estado de S\~ao Paulo (process number 2023/04651-8).}
\begin{abstract} In recent decades, the defect of finite extensions of valued fields has emerged as the main obstacle in several fundamental problems in algebraic geometry such as the local uniformization problem. Hence, it is important to identify  defectless fields and study properties related to defect. In this paper we study the relations between the following properties of valued fields: simply defectless, immediate-defectless and algebraically maximal. The main result of the paper is an example of an algebraically maximal field that admits a simple defect extension.  For this, we introduce  the notion of quasi-finite elements in the generalized power series field $k\left(\left(t^\Gamma\right)\right)$.
\end{abstract}

\keywords{Defect, henselian, quasi-finite element.}
\subjclass[2010]{Primary 13A18}

\maketitle

	\section{Introduction}

Let $(K,v)$ be a valued field, $L|K$ a finite extension of $K$ and $w_1,\dots,w_r$ the distinct extensions of $v$ to $L$. {\bf The fundamendal inequality}
\begin{equation}
\sum\limits_{i=1}^re(w_i/v)f(w_i/v)\le[L:K]\label{eq:fundamental}
\end{equation}
bounding the ramification indices and the inertia degrees of the $w_i$ in terms of $[L:K]$ was proved independently by P. Roquette and I. S. Cohen -- O. Zariski \cite{Cohen} in the nineteen fifties (all the basic notions and notation are defined in  Section \ref{Prelim} below). The {\bf defect} $d(w_i/v)$ of the extension $(L|K,w_i|v)$ of valued fields, called ``ramification deficiency'' in \cite{Cohen}, measures how far the inequality \eqref{eq:fundamental} is from equality. In some form the notion of defect was already known to A. Ostrowski in the nineteen thirties; the latter proved that $d(w_i/v)$ is always a power of the characteristic exponent $p$ of the residue field $Lv$ of the valuation ring of $v$ (that is, $p=\text{char}\ Lv$ if this is positive and $p=1$ otherwise).

\vspace{0.3cm}

An extension $(L|K,w|v)$ is said to be defectless if $d(w/v)=1$. A valued field $(K,v)$ is said to be defectless if it admits no defect extensions (that is, finite extensions $(L|K,w|v)$ with $d(w/v)>1$).

\vspace{0.3cm}

In recent decades defect, popularized by F.-V. Kuhlmann and S. D. Cutkosky, among others, has emerged as the main obstacle in several fundamental problems in algebraic geometry and model theory such as the local uniformization problem and the problem of axiomatizing the existential theory of the field $\F_p((t))$. Because of this, it is important to try to identify  defectless fields and, more generally, study properties related to defect.

\vspace{0.3cm}

If P is a property of finite field extensions, the phrase ``a P-defectless field'' will mean ``a valued field admitting no defect extensions with property P''. This paper is concerned with studying the relations between the following properties of valued fields: simply defectless, immediate-defectless and algebraically maximal.

\vspace{0.3cm}

The paper is organized as follows. In  Section \ref{Prelim} we introduce all the basic notions such as valuations, valued fields, henselization and defect. In Section \ref{SimplyDefectless} we introduce the notions of simply defectless, immediate--defectless and algebraically maximal fields and study the relationships between them. We show that simply defectless implies algebraically maximal (Proposition \ref{simplydefectlessimpliesalgmax}) and that a valued field is immediate--defectless if and only if its henselization is algebraically maximal (Proposition \ref{immdefiffhensalgmax}). The main result of this paper (Theorem \ref{teoEx}) consists in providing an example of an algebraically maximal field that is not simply defectless.  Section \ref{Main} is devoted to describing the example and proving that it has the required properties. Crucial in this construction and proofs is the newly introduced notion of a {\bf quasi-finite element} of a generalized power series field.

\section{Valued fields and defect extensions}\label{Prelim}

\begin{definition}
	Take a commutative ring $R$ with unity. A \index{Valuation}\textbf{valuation} on $R$ is a mapping $\nu:R\lra \Gamma_\infty :=\Gamma \cup\{\infty\}$ where $\Gamma$ is a totally ordered abelian group (and the extension of addition and order to $\infty$ is done in the natural way), with the following properties:
	\begin{description}
		\item[(V1)] $\nu(ab)=\nu(a)+\nu(b)$ for all $a,b\in R$.
		\item[(V2)] $\nu(a+b)\geq \min\{\nu(a),\nu(b)\}$ for all $a,b\in R$.
		\item[(V3)] $\nu(1)=0$ and $\nu(0)=\infty$.
	\end{description}
\end{definition}

Let $\nu: R \lra\Gamma_\infty$ be a valuation. The ideal $\supp(\nu)=\{a\in R\mid \nu(a )=\infty\}$ is called the \textbf{support} of $\nu$. The \textbf{value group} of $\nu$ is the subgroup of $\Gamma$ generated by 
$\{\nu(a)\mid a \in R\setminus \supp(\nu) \}$ and is denoted by $\nu R$ or $\Gamma_\nu$. A valuation $\nu$ is a \index{Valuation!Krull}\textbf{Krull valuation} if $\supp(\nu)=(0)$.  If $\nu$ is a Krull valuation, then $R$ is a domain and we can extend $\nu$ to $K={\rm Quot}(R)$ in the usual way. In this case, define the \textbf{valuation ring} as
$\VR_{K}:=\{ a\in K\mid \nu(a)\geq 0 \}$. The ring $\VR_K$ is a local ring with unique maximal ideal
$\MI_K=\{a\in K\mid\nu(a)>0 \}. $ We define the \textbf{residue field} of $\nu$ to be the field $\VR_K/\MI_K$ and denote it by $ K\nu$. The image of $a\in \VR_K$ in $ K\nu$ is denoted by $a\nu$. Throughout, we will refer to the pair $(K,v)$ as a {\bf valued field}. {\bf A valued field extensions} $(K,v)\subset (L,w)$, also denoted by $(L|K,w|v)$, is a field extension $K\hookrightarrow L$ such that
$w|_K=v$.

Let $(K,v)\subset (L,w)$ be a valued field extension. Then $vK$ can be seen as a subgroup of $wL$ and $Lw$ as a field extension $Kv$. We call
$$
e(w/v):=(wL:vK) \text{ and }  f(w/v):= [Lw: Kv]
$$
the \textbf{ramification index} and the \textbf{inertia degree}, respectively. If $n=[L:K]$ is finite, then both the ramification index and the inertia degree are finite and $e(w/v)f(w/v)\leq n$. In this situation, we can have only a finite number of extension of $v$ from $K$ to $L$. Denoting by $w_1, \ldots, w_r$ all these extensions, we have the \textit{fundamental inequality} (see \cite{Eng})
$$
\sum_{i=1}^r e(w_i/v)f(w_i/v)\leq [L:K].
$$

A valued field $(K,v)$ is said to be \textbf{henselian} if for every algebraic extension $L$ of $K$ there exists only one extension of $v$ from $K$ to $L$. Fix an algebraic closure $\overline{K}$ of K and an extension $\overline{v}$ of $v$ to
$\overline{K}$. Take the separable closure $K^{\rm sep}$ of $K$ in $\overline{K}$ and let $K^h$ be the fixed field  of $G^d:=\{\sigma\in {\rm Gal}(K^{\rm sep}\mid K)\mid \overline{v}\circ \sigma =\overline{v}\}$.  Taking the restriction $v^h=\overline{v}|_{K^h}$,  the valued field $(K^h,v^h)$ is a \textbf{henselization} of $(K,v)$, that is, a henselian field which is contained in every other henselian field that extends $(K,v)$. All henselizations of $(K,v)$ are isomorphic to each other.

\vspace{0.3cm}

Let $(K,v)\subset (L,w)$ be a finite valued field extension. Take henselizations $K^h$ and $L^h$ of $K$ and $L$ inside $\overline{K}=\overline{L}$.
\begin{definition}
	The \textbf{defect} of the extension $w/v$ is defined as
$$
d(w/v) = \frac{[L^h:K^h]}{e(w/v)f(w/v)}.
$$
\end{definition}
If we denote also by $v$ the valuation on $L$ extending $v$ on $K$, then we will write the extension as $(L\mid K, v)$ and denote the defect of the extension by $d(L\mid K, v)$.

\vspace{0.3cm}

Take $w_1, \ldots, w_r$ all extensions of $v$ to $L$. We have the following equality (see \cite{Eng}):
$$
[L:K]  = \sum_{i=1}^r e(w_i/v)f(w_i/v)d(w_i/v).
$$
If the extension $w$ is unique, then 
$$
d(w/v) = \frac{[L:K]}{e(w/v)f(w/v)} \text{ and } n=efd.
$$
The \textbf{residue characteristic} of $(K,v)$  is defined as
$$
p =\begin{cases} \Char(Kv) &\quad \text{ if } \Char(Kv)>0;\\
	1 &\quad \text{ if } \Char(Kv)=0.
\end{cases} $$
By the \textit{Lemma of Ostrowski} we know that the defect is a power of the residue characteristic. 

\vspace{0.3cm}

We say that $w/v$ is a \textbf{defect extension} if $d(w/v)>1$ and a \textbf{defectless extension} if $d(w/v)=1$. The valued field $(K,v)$ is called \textbf{defectless} if every finite valued field extension of $K$ is a defectless extension. Examples of defect extensions and defectless fields can be found in \cite{kulhmannlocalunif}.

\section{Immediate-defectless  and algebraically maximal valued fields}\label{SimplyDefectless}

\subsection{Simply defectless fields}

We will call a valued field $(K,v)$ \textbf{simply defectless} if all simple algebraic extensions of $K$ are defectless.

\vspace{0.3cm}

Every defectless field is simply defectless, but the converse is not true in general, as is shown in the example presented in \cite[Section 5]{AghighKhandujaExSimplyDefectlessNotDefectless}.

\begin{prop}\label{teoSimplyDefectlessAndDefecteless}\cite[Proposition 6.1]{CaioFCS} Assume that $(K,v)$ is henselian. If
$(K,v)$ is simply defectless, $vK$ is $p$-divisible and $Kv$ is perfect, then $(K,v)$ is a defectless field. 
\end{prop}
Simply defectless fields are connected to the studies of \textit{(abstract) key polynomials} and \textit{complete sequences}, as shows the next result. For more details on key polynomials  see \cite{spivamahboubkeypoly} and \cite{josneiKeyPolyPropriedades}.

\begin{prop}\label{corFCS1simplydefectless}\cite[Corollary 6.2]{CaioFCS} We have $(K,v)$ simply defectless and henselian if and only if every simple algebraic valued field extension of $K$ admits a finite complete sequence of key polynomials.
\end{prop}

\subsection{Algebraically maximal and immediate-defectless  fields}

A valued field extension $( L\mid  K, v)$ is said to be \textbf{immediate} if $vL=v K$ and $Lv=Kv$. 

\begin{definition} The valued field $(K,v)$ is said to be {\bf algebraically maximal} if it admits no proper algebraic immediate extensions.
\end{definition}

\begin{remark} The valued field $(K,v)$ is algebraically maximal if and only if it admits no {\it simple} algebraic immediate extensions.
\end{remark}

Every algebraically maximal field is henselian, since the henselization of a valued field is an immediate extension. The converse is not true in general (see \cite[Theorem 3.7]{kulhmannlocalunif}). Every henselian defectless field is algebraically maximal, but the converse does not hold in general (see \cite[Theorem 3.26]{kulhmannlocalunif}). 

\begin{remark} A field $K$ is algebraically maximal if and only if for every $x$ algebraic over $K$ the set $\{v(x-a)\ |\ a\in
K\}$ contains a maximal element (this is the (i) $\iff$ (ii) implication of Theorem 1.1 of \cite{BK}). The paper \cite{BK} contains other interesting information on comparing the algebraically maximal condition with the existence and non-existence of simple defect extensions.
\end{remark}
\medskip

%

\begin{definition} The valued field $(K,v)$ is called {\bf immediate-defectless} if it admits no immediate defect extensions.
\end{definition}

\begin{remark}
	An algebraically maximal field is clearly immediate-defectless. For henselian fields the converse is also true: a henselian field is immediate-defectless if and only if it is algebraically maximal.
\end{remark}

\noindent{\bf Notation.} Fix an extension of $v$ to the algebraic closure $\bar K$ of $K$ and let $(K^h,v^h)$ be the corresponding henselization.

\begin{prop}\label{simplydefectlessimpliesalgmax} Assume that $(K,v)$ is simply defectless. Then it is immediate-defectless.
\end{prop}
\begin{proof} Since henselization is immediate and every finite subextension of $K^h$ is defectless, replacing
$(K,v)$ by $(K^h,v^h)$ does not change the problem. We will assume that $K$ is henselian and, in particular, ``immediate--defectless'' is synonymous with ``algebraically maximal''.
	\medskip
	
	Consider a non-trivial immediate extension $(L|K,v)$, aiming for contradiction. Take a simple sub-extension $K(x)\subset L$ with $x\notin K$. Then $(K(x)|K,v)$ is immediate, hence a defect extension (since $K$ is henselian). This gives the desired contradiction.
\end{proof}

%

\begin{prop}\label{immdefiffhensalgmax}
	The field $(K,v)$ is immediate-defectless if and only if $\left(K^h,v^h\right)$ is algebraically maximal. 
\end{prop}

\begin{proof}
	Assume that $(K,v)$ is immediate--defectless and that $(L,w)$ is a non-trivial simple algebraic immediate extension of
$\left(K^h,v^h\right)$. Then $\left(L\mid K^h, w\mid v^h\right)$ has defect (because $v^h$ is the unique extension of $w$, since
$\left(K^h,v^h\right)$ is henselian). Write $L=\frac{K^h[x]}{(g)}$, where $g\in K^h[x]$  is a polynomial generating a maximal ideal $(g)\subset K^h[x]$. Let $K_0$ denote a finite extension of $K$, contained in $K^h$ and containing all the coefficients of $g$. Let
$L_0:=\frac{K_0[x]}{(g)}$ and $w_0=w|_{L_0}$. Then $L=L_0^h$. Since the extension $\left(L\mid K^h, w\mid v^h\right)$ has defect, so does $\left(L_0\mid K, w_0\mid v\right)$ (by definition of defect). Since both $\left(L\mid K^h, w\mid v^h\right)$ and $\left(\left.K^h\right|K,\left.v^h\right|v\right)$ are immediate, we conclude that
$(L_0|K,w|v)$ is an immediate defect extension, a contradiction. 
	
	Next, assume that  $(K^h,v^h)$ is algebraically maximal and that $(L, w)$ is an immediate defect extension of $(K,v)$.  Consider the henselization $(L^h, w^h)$ of $(L,w)$. 
	We have
	\[
	(K,v) \subseteq (L,w) \subseteq (L^h, w^h).
	\]
	By definition of henselization, we have $\left(K^h,v^h\right)\subseteq\left(L^h, w^h\right)$. Since $(L|K, w|v)$ has defect, we must have $K^h\neq L^h$ (by definition of defect). Hence, $(L^h, w^h)$ is a proper extension of $(K^h, v^h)$. Since $(L|K,w|v)$ is immediate and henselizations are immediate extensions, $w^hL^h=v^hK^h$ and $L^hv^h=K^hv^h$ and then $(L^h,w^h)$ is an immediate extension of $(K^h,v^h)$, a contradiction. 
\end{proof}

\subsection{A defectless field that is not algebraically maximal}

If $(K,v)$ is not henselian, then $(K,v)$ simply defectless implies that it is immediate-defectless but not that it is algebraically maximal. For example, if $k$ is a field and $t$ an independent variable then $k(t)$ with the $t$-adic valuation $v$ is defectless, but not algebraically maximal, as we now show.
\smallskip

The fact that $k(t)$ with the $t$-adic valuation is defectless  follows from the more general Theorem 5.1 of \cite{kulhmannlocalunif}. We give a direct proof in order to be self-contained. For simplicity, we will assume that $k$ is perfect. The case of an arbitrary $k$ can be easily reduced to the perfect one by considering the prime field of $k$, but we omit the details.
\smallskip

We will need two lemmas. Let $K:=k\langle t\rangle$ be the henselization  of $k(t)$

\begin{lemma} Let $(L|K,v)$ be a finite extension of valued fields. Then there exists a finite extension $k\subset k'$ and an element $t'\in L$ such that $L\cong k'\left\langle t'\right\rangle$ endowed with the $t'$-adic valuation.
\end{lemma}
\begin{proof} Let $k'=Lv$. Since $k$ is perfect, the field extension $k'|k$ is separable, hence simple by the Primitive Element Theorem. Take an element $\bar a\in k'$ such that
	\[
	k'=k\left(\bar a\right).
	\]
	Now, $\bar a$ satisfies a separable algebraic equation $f$ over $k$. By Hensel's Lemma, there exists $a\in L$ such that $\bar a=av$ and $f(a)=0$. We have obtained an embedding $k'\cong k(a)\subset L$.
	
	Make the identifications $vK\cong\Z$ and $vL\cong\frac1e\Z$, where $e=[vL:vK]$. Take an element $t'\in L$ such that $vt'=\frac1e$. Then $k'\left\langle t'\right\rangle\subset L$ (since $L$ is henselian). Since both $L$ and $k'\left\langle t'\right\rangle$ have transcendence degree 1 over $k$, the  extension $L\left|k'\left\langle t'\right\rangle\right.$ is algebraic. Moreover, the algebraic extension $\left(L\left|k'\left\langle t'\right\rangle\right.,v\right)$ of henselian valued fields is immediate, hence an isomorphism. This completes the proof of the lemma.
\end{proof}

\begin{lemma}\label{L|Kdefectless} Let $(L|K,v)$ be a finite extension of valued fields. Then $(L|K,v)$ is defectless.
\end{lemma}
\begin{proof} We have the following cases to consider.  
	\smallskip 
	
	\noindent  {\bf Case 1.} If $p\ \nmid\ [L:K]$, then the extension is defectless by Ostrowski's Lemma.
	\smallskip
	
	\noindent{\bf Case 2.} Assume that $[L:K]=p$. If $(L|K,v)$ were a defect extension, then it would be an immediate algebraic extension of the henselian field $K$, a contradiction. This proves the lemma in Case 2.
	\smallskip
	
	\noindent{\bf Case 3.} The general case. Enlarging $L$, if necessary, we may assume that the extension $L|K$ is normal. We proceed by induction on $[L:K]$, where we consider Cases 1 and 2 to be the base of the induction (which includes all the cases where $[L:K]$ is prime).
	
	Assume that $d:=[L:K]$ is a composite number divisible by $p$ and that the lemma is true for all the normal extensions of degree strictly smaller than $d$. Let $G={\rm Aut}(L/K)$.
	\smallskip
	
	\noindent{\bf Case 3a.} There exists an intermediate field $K'$ with $K\subsetneqq K'\subsetneqq L$ such that
	$\left[K':K\right]=p$ (this happens, in particular, whenever $L|K$ is inseparable).
	\smallskip
	
	\noindent{\bf Case 3b.} The extension $L|K$ admits no subextension of degree $p$ over $K$. In particular, $d$ is not a power of $p$ and $L|K$ is separable, so $|G|=d$. Since $p\ |\ d$ but $d$ is not a power of $p$, $G$ contains a proper $p$-Sylow subgroup $H$. Let $K'$ be the fixed field of $H$ (which need not be a normal extension of $K$).
	\smallskip
	
	In Case 3a, the extension $\left(\left.K'\right|K,v\right)$ is defectelss by Case 2. In case 3b, we have $p\ \nmid\ \left[K':K\right]$, so, again, the extension $\left(\left.K'\right|K,v\right)$ is defectless. In all cases, the normal extension
	$\left(L\left|K',v\right.\right)$ is defectless by the induction assumption, hence so is  $(L|K,v)$. This completes the proof of the lemma.
\end{proof}

\begin{prop} The valued field $k(t)$ is a defectless field that is not algebraically maximal.
\end{prop}
\begin{proof} By Lemma \ref{L|Kdefectless}, the henselization $K=k\langle t\rangle$ of $k(t)$ is defectless, hence so is
$k(t)$ by definition of defect. The result follows immediately from the fact that the henselization $K$ is a proper immediate extension of $k(t)$.

\end{proof}

\section{An algebraically maximal field that is not simply defectless}\label{Main}

The converse of Proposition \ref{simplydefectlessimpliesalgmax} does not hold in general. Indeed, in this section we will construct an example of an algebraically maximal field that is not simply defectless.
We start with some preliminaries.
\\


Let $p,q$ be two distinct prime numbers, $k$ a field of characteristic $p$, $t$ an independent variable and $\Gamma$ a subgroup of $\Q$ containing $\frac1{q^i}$ for all $i\in\N$. We work with valued subfields of the valued field $k\left(\left(t^{\frac1p\Gamma}\right)\right)$ endowed with the $t$-adic valuation $v$. For $i\in\N_{>0}$, put $\beta_i=1-\frac1{q^i}$. The sequence $(\beta_i)_{i\in\N}$ is an increasing sequence  of elements of $\Gamma$ such that
\begin{equation}
	\sup\limits_{i\in\N}\{\beta_i\}=1.
\end{equation}
Consider the generalized power series $w=\sum\limits_{i=1}^\infty t^{\beta_i}\in k\left(\left(t^\Gamma\right)\right)$. For $i\in\N_{>0}$, write $w=w_{0i}+w_i$, where
\begin{eqnarray}
	w_{0i}&=&\sum\limits_{j=1}^it^{\beta_j}\quad\text{and}\\
	w_i&=&\sum\limits_{j=i+1}^\infty t^{\beta_j};
\end{eqnarray}
we have $vw_i>0$ for all $i\in\N_{>0}$.
\medskip

\noindent{\bf Notation.} For an element $b\in\N$, let $\partial_b$ denote the $b$-th Hasse derivative with respect to $W$. 
\medskip

Let $n$ be a non-negative integer and $\gamma_1,\dots,\gamma_n$ strictly positive elements of $\Gamma$. Let
$\Gamma^*\supset\Gamma$ be another ordered abelian group and $\beta\ge1$ an element of $\Gamma^*$. Let $\mathcal O$ denote the valuation ring associated to  the $t$-adic valuation of $k\left(t^{\gamma_1},\dots,t^{\gamma_n}\right)$ and $\tilde{\mathcal O}$ --- the henselization of $\mathcal O$. Take elements $a_i\in\tilde{\mathcal O}$, $i\in\N$, such that not all of the $a_i$ are equal to 0, and an element $c\in k$.  Let $W$ be an independent variable and consider
\[
f=\sum\limits_{i\in\N}a_iW^i\in\tilde{\mathcal O}\langle W\rangle\setminus\{0\},
\]
where $\tilde{\mathcal O}\langle W\rangle$ denotes the common henselization of the rings $\tilde{\mathcal O}\left[W\right]$ and
$\mathcal O\left[W\right]$.
\begin{prop}\label{truncation} There exist $l_0\in\N_{>0}$ and $e\in\N$ such that for all integers $l\ge l_0$ the following conditions hold.
	\begin{enumerate}
		\item
		\[
		v\left(f\left(w+ct^\beta\right)\right)=v(f(w_{0l}))\ne\infty.
		\]
		\item
		\begin{eqnarray*}
			v(f(w_{0l}))<&\min\limits_{i\in\N_{>0}}&\left\{v\left(\partial_if(w)\right)+iv(w_l)\right\}=\\
			&\min\limits_{i\in\N_{>0}}&\left\{v\left(\partial_if(w)\right)+i\beta_{l+1}\right\}.
		\end{eqnarray*}
		\item If $f\notin\tilde{\mathcal O}$, then
		\[
		v\left(\partial_{p^e}f\left(w+ct^\beta\right)\right)+p^e\beta_{p^e+1}<
		v\left(\partial_if\left(w+ct^\beta\right)\right)+i\beta_{l+1}\ \forall i\in\N_{>0}\setminus\{p^e\}.
		\]
	\end{enumerate}
\end{prop}
\begin{proof} First, assume that $f\in\tilde{\mathcal O}[W]$ and let $d$ denote the degree of $f$ as a polynomial in $W$. We proceed by induction on $d$, the case $d=0$ being trivial. Assume the result is true up to degree $d-1$; in particular, it holds for all the derivatives $\partial_if(W)$ of $f$ with $i\in\N_{>0}$.
	
	Fix $l_0$ sufficiently large so that the subgroup $(\gamma_1,\dots,\gamma_n)$ does not contain $\frac1{q^{l_0}}$ and so that the conclusion of the proposition holds with $f$ replaced by $\partial_if(W)$, $i\in\N_{>0}$. In particular, for all $i\in\N_{>0}$ and all $l\ge l_0$ we have
	\[
	v\left(\partial_if\left(w+ct^\beta\right)\right)=v(\partial_if(w_{0l})).
	\]
	By Taylor's formula, for every $l\in\N_{>0}$ we have 
	\[
	f\left(w+ct^\beta\right)=f(w_{0l})+\sum\limits_{i\in\N_{>0}}\partial_if(w_{0l})\left(w_l+ct^\beta\right)^i.
	\]
	\begin{remark}\label{distinctslopes} Before proceeding, we make a remark about basic properties of linear functions on the real line. Let $\ell_1(\epsilon)=a_1+b_1\epsilon$ and $\ell_2(\epsilon)=a_2+b_2\epsilon$, where $a_j,b_j\in\R$, be two linear functions in one variable. Assume that $b_1\ne b_2$. Then there exists a unique $\epsilon\in\R$ such that $\ell_1(\epsilon)=\ell_2(\epsilon)$. In particular, if $(\theta_1,\theta_2)\subset\R$ is an {\it open} interval (where we allow $\theta_1=-\infty$ and/or $\theta_2=+\infty$) then there exists $\beta\in(\theta_1,\theta_2)$ such that $\ell_1(\beta')\ne\ell_2(\beta')$ for all $\beta'\in(\beta,\theta_2)$. By the same token, if $\ell_i(\epsilon)=a_i+b_i\epsilon$, $i\in\{1,\dots,d\}$, are $d$ linear functions whose slopes $b_1,\dots,b_d$ are pairwise distinct, we have
			\[
			\#\left\{\epsilon\in\R\ \left|\ \exists i,j\in\{1,\dots,d\},i\ne j,\text{ such that
			}\ell_i(\epsilon)=\ell_j(\epsilon)\right.\right\}\le\frac{d(d-1)}2.
			\]
			In particular, if $(\theta_1,\theta_2)\subset\R$ is an open interval then there exists $\beta\in(\theta_1,\theta_2)$ such that $\ell_i(\beta')\ne\ell_j(\beta')$ whenever $i\ne j$ and $\beta'\in(\beta,\theta_2)$.
		\end{remark}
		With this in mind, we come back to the proof of the proposition.
	\medskip
	
	Consider linear functions
	$\ell_i(\epsilon):=v\left(\partial_if(w_{0l})\right)+i\epsilon$, $\epsilon\in(0,1)$, $i\in\{1,\dots,d\}$. By Remark \ref{distinctslopes}, for $\epsilon$ sufficiently close to 1, all these linear functions have distinct values. Therefore,
	increasing $l_0$, if necessary, we may assume that
	\[
	\ell_i(\beta_j)\ne\ell_{i'}(\beta_j)\ \text{ whenever }\ i\ne i'\ \text{ and }\ j\ge l_0.
	\]
	Let us prove part (3) of the proposition. In order to do this, take a $j\in\N_{>0}$ that is not a power of $p$. Then $j$ can be written as $j=j'+j''$ with $\binom j{j'}\ne0$: if $j=p^\ell u$ with $u\ge2$ not divisible by $p$, we can take $j'=p^\ell$ and $j''=(u-1)p^\ell$. We have $\binom j{j'}\partial_j=\partial_{j'}\circ\partial_{j''}$. By (2) of the proposition, applied to $\partial_{j'}f(W)$ instead of $f$, we see that
		\[
		v\left(\partial_{j'}f\left(w_{0l}\right)\right)<v\left(\partial_jf\left(w_{0l}\right)\right)+j''\beta_{l+1}.
		\]
		Adding $j'\beta_{l+1}$ to both sides, we obtain
		\[
		v\left(\partial_{j'}f\left(w_{0l}\right)\right)+j'\beta_{l+1}<v\left(\partial_jf\left(w_{0l}\right)\right)+j\beta_{l+1}.
		\]
		This proves that 
		\begin{equation}
			v\left(\partial_jf\left(w\right)\right)+j\beta_{l+1}\ne
			\min\limits_{i\in\N_{>0}}\left\{v\left(\partial_if(w)\right)+i\beta_{l+1}\right\}.\label{eq:minimumppower}
		\end{equation}
		Thus the minimum on the right hand side of \eqref{eq:minimumppower} can be attained only if $i$ is a power of $p$. Since the elements $v\left(\partial_if(w_{0l})\right)+i\beta_{l+1}$, $i\in\N_{>0}$, are pairwise distinct, this proves  part (3) of the proposition for polynomials $f$ of degree $d$.
		\medskip
		
		To prove (1) and (2) of the proposition, let $e\in\N$ be as in part (3). We have
		\begin{eqnarray*}
			v(f(w_{0l}))&\in&\left(\gamma_1,\dots,\gamma_n,\frac1{q^l}\right)\bigcup\{\infty\},\\
			v\left(\partial_{p^e}f(w_{0l})\right)&\in&\left(\gamma_1,\dots,\gamma_n,\frac1{q^l}\right)\text{ and}\\
			v\left(\partial_{p^e}f(w_{0l})\right)+p^e\beta_{l+1}&\notin&\left(\gamma_1,\dots,\gamma_n,\frac1{q^l}\right).
		\end{eqnarray*}
		Thus 
		\[
		v\left( f(w_{0l})\right)\ne v\left(\partial_{p^e}f(w_{0l})\right)+p^e\beta_{l+1},
		\]
		so
	\begin{eqnarray}
		v\left(f\left(w_l+ct^\beta\right)\right)=&v\left(f(w_{0l})+\sum\limits_{i\in\N_{>0}}\partial_if(w)\left(w_l+ct^\beta\right)^i\right)\nonumber\\
		=&\min\left\{v\left(f(w_{0l})\right),v\left(\partial_{p^e}f(w)\right)+p^e\beta_{l+1}\right\}.\label{eq:taylor}
	\end{eqnarray}
In particular $v(f(w_l+ct^\beta))\ne\infty$, in other words, $f(w_l+ct^\beta)\ne0$. Now, the sequence
		$(v\left(\partial_{p^e}f(w_{0l})\right)+p^e\beta_{l+1})_l$ is strictly increasing with $l$ whereas $v\left(f\left(w+ct^\beta\right)\right)$ is independent of $l$. Both (1) and (2) of the proposition follow from this.
	
	Next, drop the assumption that $f$ is a polynomial in $W$. Take the smallest $d\in\N$ such that $va_d\le va_i$ for all $i\in\N$. Replacing $f$ by $\frac f{a_i}$ does not change the problem, so we may assume that $a_i=1$. By the henselian Weierstrass Preparation Theorem \cite[Proposition 3.1.2]{BC}, we can write $f=u\tilde f$, where $\tilde f$ is a polynomial of degree $d$ with coefficients in the maximal ideal of $\tilde{\mathcal O}\langle W\rangle$  and $u$ is a unit of $\tilde{\mathcal O}\langle W\rangle$. Since the proposition is already known for $\tilde f$, we obtain that $f\left(w+ct^\beta\right)\ne0$ and there exists $l_0\in\N$ such that for all $l\ge l_0$ we have $v(f\left(w+ct^\beta\right))=v(f\left(w_{0l}\right))$. Take an integer $i\in\N_{>0}$ such that
		\begin{equation}
		a_{ i}\ne0,\label{eq:non0derivative}
		\end{equation}
		so that $\partial_if(W)\ne0$. Applying the above reasoning to $\partial_if$ instead of $f$, we obtain that
		$\partial_if\left(w+ct^\beta\right)\ne0$ and there exists $l_0\in\N$ such that
		\begin{equation}
			\forall l\ge l_0\text{ we have }v\left(\partial_if\left(w+ct^\beta\right)\right)=v\left(\partial_if\left(w_{0l}\right)\right).\label{eq:dfistabilizes}
		\end{equation}
		Fix an $i\in\N$ satisfying \eqref{eq:non0derivative} and an $l_0\in\N$ satisfying \eqref{eq:dfistabilizes}. Let $i_1\in\N$ be sufficiently large so that
		\begin{equation}
			(i_1-i)\beta_1>v\left(\partial_if\left(w+ct^\beta\right)\right).\label{eq:i1verylarge}
		\end{equation}
		For $i'\in\N_{>0}$, $\partial_{i'}f(W)$ is a formal power series over $\tilde{\mathcal O}$, so $v\left(\partial_{i'}f(w_{0l})\right)\ge0$. Thus the inequality \eqref{eq:i1verylarge} implies that for every $i'>i_1$ and every $l\ge l_0$ we have
\[
v\left(\partial_if(w_{0l})\right)+i\beta_{l+1}<v\left(\partial_{i'}f(w_{0l})\right)+i'\beta_{l+1}.
\]
In particular,
		\[
		\min\limits_{i'\in\N_{>0}}\left\{v\left(\partial_{i'}f(w_{0l})\right)+i'\beta_{l+1}\right\}
		\]
		cannot be attained for any $i>i_1$.
		
		Now (1)--(3) of the proposition are proved exactly as in the polynomial case, except that at every step we restrict attention to the quantities
		\[
		v\left(\partial_{i'}f(w_{0l})\right)+i'\beta_{l+1}\quad\text{ for }\quad i'\le i_1.
		\]
	 This completes the proof.
\end{proof}

\begin{remark}\label{perturbation} The above  proposition holds if we replace $w$ in (1), (2) and (3) by $w_r$ for some $r\in\N$. The same proof applies, where, naturally, we take $l>l_0\ge r$.
\end{remark}
Let $s=w^{\frac1p}$.
\begin{definition} An element $y\in k\left(\left(t^\Gamma\right)\right)$ is said to be $w${\bf-quasi-finite} if there exist
	$\gamma\in\Gamma_{\le0}$, $n\in\N$, elements $h_1,\dots,h_n\in k\left(t^\Gamma\right)[w]$ having strictly positive values and a formal power series $g$ in $n$ variables over $k$ such that $y=t^\gamma g\left(h_1,\dots,h_n\right)$.
	Similarly, consider an ordered abelian group $\Gamma^*\supset\frac1p\Gamma$, an element $\beta\in\Gamma^*$,
	$\beta>\frac1p$, and a constant $c\in k$. An element $y\in k\left(\left(t^{\Gamma^*}\right)\right)$ is said to be
	$(s+ ct^\beta)${\bf-quasi-finite} if there exist $\gamma\in\frac1p\Gamma_{\le0}$, $n\in\N$, elements
	$h_1,\dots,h_n\in k\left(t^{\frac1p\Gamma}\right)[s+ ct^\beta]$ having strictly positive values and a formal power series $g$ in $n$ variables over $k$ such that $y=t^\gamma g\left(h_1,\dots,h_n\right)$.
\end{definition}
\begin{example} Every element $y\in k\left(t^\Gamma,w\right)$ is $w$-quasi-finite. For example, let us write $w^{-1}$ as a $w$-quasi-finite element. We have
	\begin{align*}
		\frac{1}{w}&=\frac{1}{t^{\beta_1}(1+t^{\beta_2-\beta_1}+t^{\beta_3-\beta_1}+\dots)} \\
		&=\frac{1}{t^{\beta_1}(1+z)}, \text{ where } z=t^{-\beta_1}w-1=t^{-\beta_1}w_{2}\\
		&= t^{-\beta_1}(1-z+z^2-z^3+\ldots)
	\end{align*}
	Hence,  we take $\gamma=-\beta_1$, $h_1=z\in k\left(t^\Gamma\right)[w]$ and $g(X)=\sum\limits_{n=0}^\infty (-1)^nX^n$ and we see that $w^{-1}=t^\gamma g(h_1)$.
\end{example}

\begin{remark}\label{quasifinitepowerseries} Let $h$ be a $w$-quasi-finite element. Then there exist
	$r\in\N_{>0}$, $n\in\N$, elements $\gamma_1,\gamma_2,\dots,\gamma_n\in\Gamma_{>0}$, $\gamma\in\Gamma_{\ge0}$, $\beta\in\Gamma$ with $0\le\beta<\beta_{r+1}$ such that $t^\gamma h\in
	k\left[t^{\gamma_1},\dots,t^{\gamma_n}\right]\left[\left[\frac{w_r}{t^\beta}\right]\right]$. For example, for $h=w^{-1}$ we can take $r=2$, $n=0$, $\gamma=\beta=\beta_1$ and $g(X)=\sum\limits_{n=0}^\infty (-1)^nX^n$, as above. We have
\[
t^\gamma w^{-1}=g\left(\frac{w_2}{t^{\beta_1}}\right)\in k\left[\left[\frac{w_2}{t^{\beta_1}}\right]\right].
\]
\end{remark}

Let $K=k\left\langle t^\Gamma,w\right\rangle$ denote the henselization of $k\left(t^\Gamma,w\right)$. 

\begin{remark}
		Recall that $K$ is the smallest henselian extension of $k\left(t^\Gamma,w\right)$, that is, the smallest extension whose valuation ring $\mathcal O_K$ satisfies Hensel's lemma (every polynomial over $\mathcal O_K$ having a root modulo $\mathfrak m_K$ has a root in $\mathcal O_K$). It is known  that rank 1 complete valued fields are henselian, hence $K\subset k\left(\left(t^\Gamma\right)\right)$. The henselization is obtained by adjoining all the elements $y$ of $k\left(\left(t^\Gamma\right)\right)$ that satisfy an equation $f$ over $k\left(t^\Gamma,w\right)$ such that $yv$ is a simple root of the reduction $\bar f$ of $f$ modulo $\mathfrak m_K$. Conversely, every element $y\in K$ with $vy=0$ satisfies an equation $f$ as above. It is also known that henselization coincides with the relative {\it separable} algebraic closure of $k\left(t^\Gamma,w\right)$ inside $k\left(\left(t^\Gamma\right)\right)$ when  $\text{rk}\ \Gamma=1$.
\end{remark}

\begin{lemma}\label{henselizationquasifinite} 
Every element $y\in K$ is $w$-quasi-finite.
\end{lemma}

\begin{proof} Multiplying $y$ by an element of the form $t^\gamma$ does not change the problem, so we may assume that $vy=0$ and apply the above description of $y$ as a root of an equation $f$ with certain properties. Let us write $f$ as $f=(Y-a)h(Y)+r(Y)$, where $va=0$, $h(Y)$ is a non-zero polynomial over $k\left(t^\Gamma,w\right)$ all of whose non-zero coefficients have value 0, the reduction $\bar h$ of $h$ modulo
	$\mathfrak m_K$ is not divisible by $Y-av$ and
	\[
	r(Y)=\sum\limits_{j=0}^mr_jY^j\text{ with }vr_j>0\text{ for all }j\in\{0,\dots,m\}.
	\]
	Write $h(Y)=(Y-a)\sum\limits_{j=0}^eh_jY^j+c$ with $c\in k\left(t^\Gamma,w\right)$, $vc=0$. Finally, write $a=\bar
	a+\tilde a$, $c=\bar c+\tilde c$ and $h_j=\bar h_j+\tilde h_j$ where $\bar a,\bar c,\bar h_j\in k$ and $\tilde a,\tilde c,\tilde h_j$ have strictly positive values. Then the equation $f(y)=0$ yields a recursive procedure for writing $y$ as a formal power series over $k$ in $\tilde a,\tilde c,\tilde h_0,\dots,\tilde h_a$ and $r_0,\dots,r_m$. Now, $\tilde a,\tilde c,\tilde
	h_0,\dots,\tilde h_a$ and $r_0,\dots,r_m$ are elements of $k\left(t^\Gamma,w\right)$ and hence are themselves
	$w$-quasi-finite; therefore, so is $y$, as desired.
	\end{proof}
\medskip

\begin{corollary}\label{coryEqualsum} Every element $y\in K$ can be written in the form
\begin{equation}
y=\sum\limits_{(i,j)\in\N^2}c_{ij}t^{\epsilon_i}\left(\frac{w_r}{t^\beta}\right)^j,\quad
\epsilon_i\in\left(\sum\limits_{i=1}^n\N\gamma_i\right)-\gamma,\quad c_{ij}\in k\label{eq:bounded->finite}
\end{equation}
(where $\gamma$, $\gamma_i$ and $\beta$ are as in Remark \ref{quasifinitepowerseries}).
\end{corollary}
\begin{remark}\label{bounded->finite} Let the notation be as in the above remark and take an element
$\epsilon\in\Gamma$. Then $\#\left\{(i,j)\in\N^2\ \left|\
\epsilon_i+jv\left(\frac{w_r}{t^\beta}\right)<\epsilon\right.\right\}<\infty$. In other words, the set of terms on the right hand side of \eqref{eq:bounded->finite} whose value is bounded above  by any given element $\epsilon\in\Gamma$ is a finite set. This fact will be used in our proof of the main theorem.
\end{remark}

Keep the above notation. Assume, in addition, that $k$ is algebraically closed. Take $\Gamma$ to be the group consisting of all the rational numbers whose denominators are not divisible by $p$; in particular, $\frac1p\notin\Gamma$.

\begin{theorem}\label{teoEx}
	The field $K=k\left\langle t^\Gamma,w\right\rangle$ is  an algebraically maximal field that is not simply defectless. 
\end{theorem}

In order to prove that the field $K$ is algebraically maximal, we will use the following lemma.

	\begin{lemma}\label{subtractppowers} Take an element
		\[
		z\in K\setminus K^p.
		\]
		There exists $a\in k\left(t^\Gamma\right)[w]\subset K$ such that $v(z-a^p)\notin p\Gamma$.
\end{lemma}
\begin{proof} The element $z$ is $w$-quasi-finite; apply Remark \ref  {quasifinitepowerseries} to $z$. Let $r\in\N_{>0}$, $n\in\N$, elements
		$\gamma_1,\gamma_2,\dots,\gamma_n\in\Gamma_{>0}$, $\gamma\in\Gamma_{\ge0}$, $\beta\in\Gamma$ with $0\le\beta<\beta_{r+1}$ and
		\[
		f\in k\left[t^{\gamma_1},\dots,t^{\gamma_n}\right]\left[\left[W\right]\right]
		\]
		be such that 
		\[
		t^\gamma z=f\left(\frac{w_r}{t^\beta}\right),
		\]
		as in Remark \ref {quasifinitepowerseries}. Multiplying $z$ by an element of $K^p$ does not change the problem, so we may assume that $\gamma=0$. Let $\Gamma_0$ denote the subgroup of $\Gamma$ generated by
		$\gamma_1,\dots,\gamma_n$. Write $z$ as
		\begin{equation}
			z=f\left(\frac{w_r}{t^\beta}\right)=\sum\limits_{(i,j)\in S}b_{ij}t^{\epsilon_i}\left(\frac{w_r}{t^\beta}\right)^j,\quad b_{ij}\in k,\epsilon_i\in(\Gamma_0)_{\ge0},S\subset\N^2.\label{eq:monomialexpansion}
		\end{equation}
		First, assume that
		\begin{equation}
			f\in k\left[t^{\gamma_1},\dots,t^{\gamma_n}\right]\left[W\right],\label{eq:assumefpolyn}
		\end{equation}
		so that the set $S$ in \eqref{eq:monomialexpansion} is finite. Subtracting an element of $\left(k\left(t^\Gamma\right)[w]\right)^p\subset K^p$ from $z$, if necessary, we may assume that
	\begin{equation}
		p\ \nmid\ (\epsilon_i,j)\ \text{ for all} \ (i,j)\in S.\label{eq:pdoesnotdivide}
	\end{equation}
	Apply Proposition \ref{truncation} to the expression \eqref{eq:monomialexpansion}, with $\frac{w_r}{t^\beta}$ playing the role of $w+ct^\beta$. Increasing $r$, if necessary, assume that the conclusion of Proposition \ref{truncation} holds with $l_0=r$ and that
		\begin{equation}
			\gamma_1,\dots,\gamma_n,\beta\in\frac1{q^r}\Z.\label{eq:in1ql0}
		\end{equation}
		Let $e$ be as in Proposition \ref{truncation} (3). Since $\beta_{r+1}\notin\frac1{q^r}\Z$ and in view of \eqref{eq:in1ql0}, we see that the initial monomial $t^\gamma$ of
		$\left(\partial_{p^e}f\left(\frac{w_{0r}}{t^\beta}\right)\right)\left(\frac{w_r}{t^\beta}\right)^{p^e}$ cannot be canceled by any of the monomials appearing in $f\left(\frac{w_{0r}}{t^\beta}\right)$. By \eqref{eq:pdoesnotdivide}, we have
		$\gamma\notin p\Gamma$. Let $b$ denote the sum of all the monomials in $t$ appearing in
		$f\left(\frac{w_{0r}}{t^\beta}\right)$ whose exponents are divisible by $p$ and put $a=b^{\frac1p}$. Then
	$a\in k\left(t^\Gamma\right)$ and $v(z-a^p)\notin pvK$, as desired.
		
		Next, drop the assumption \eqref{eq:assumefpolyn}, so that the set $S$ is allowed to be infinite in \eqref{eq:monomialexpansion}. Let $i_1\in\N$ be such that
		$\left(i_1-p^e\right)(\beta_1-\beta)>\partial_{p^e}f\left(\frac{w_{0r}}{t^\beta}\right)$, so that
\[
\min\limits_{i'\in\N_{>0}}\left\{v\left(\partial_{i'}f(w_{0r})\right)+i'(\beta_{r+1}-\beta)\right\}
\]
cannot be attained for any $i'>i_1$. Now the above proof of the lemma applies verbatim, except that statement \eqref{eq:pdoesnotdivide} should be modified to say
		\begin{equation}
			p\ \nmid\ (\epsilon_i,j)\text{ for} \ (i,j)\in S\text{ such that }\epsilon_i+j(\beta_{r+1}-\beta)\le
			v\left(\partial_{p^e}f(w_{0r})\right)+p^e(\beta_{r+1}-\beta).\label{eq:pdoesnotdividegeneralcase}
		\end{equation}
		This completes the proof.
\end{proof}

\begin{proof}[Proof of Theorem \ref{teoEx}] We will work with valued subfields of the valued field\linebreak$k\left(\left(t^{\frac1{p^2}\Gamma}\right)\right)$ of generalized series over $k$ with exponents in
$\frac1{p^2}\Gamma$, endowed with the $t$-adic valuation.
	
	First, let us show that $K$ is not simply defectless. Let
	\[
	L=K(x)=\frac{K[X]}{\left(X^{p^2}-t^{p+1}-w^p\right)}.
	\]
	We have $[L:K]=p^2$; the extension $(L|K,v)$ can be seen as a tower
	\[
	K\subset K'\subset L
	\]
	of two purely inseparable extensions of degree $p$ with $K'=K(y)=K\left(t^{\frac1p}\right)$ and
	\[
	y=\left(t^{p+1}+w^p\right)^{\frac1p}.
	\] 
	Now, $y':=y-w\in K'$ and $vy'=\frac{p+1}p$, so $\Gamma':=vK'=\Gamma+\frac1p\Z=vK+\frac1p\Z$ and
	$\left[vK':vK\right]=p$. We claim that $(L|K',v)$ is immediate (and hence both it and $(L|K,v)$ are defect extensions). We have $L=\frac{K'[X]}{\left(X^p-y\right)}$.
	
	Take an element $z\in L$. Write $z=f(x)$, where
	\[
	f(X)=\sum\limits_{j=0}^{p-1}b_jX^j\quad\text{ with }\ b_j\in K'.
	\]
	To prove that  $\left(L|K',v\right)$ is immediate, it is sufficient to prove that
	\begin{equation}
		v(f(x))=v\left(f\left(s+t^{\frac{p+1}{p^2}}\right)\right)\in vK'.
	\end{equation}
	
	For $r\in\N_{>0}$, let $s_r=w_r^{\frac1p}$.
	
	Since $z=f(x)\in k\left\langle t^{\frac1p\Gamma},x\right\rangle$, it is an $x$-quasi-finite element of $k\left(\left(t^{\frac1p\Gamma}\right)\right)$ (by the same argument as in Lemma \ref{henselizationquasifinite}). By Remark \ref{quasifinitepowerseries}, applied to $s+t^{\frac{p+1}{p^2}}$ instead of $w$, we have
	\[
	f(s)=t^\gamma g\left(\frac{s_r}{t^\beta}+t^{\frac{p+1}{p^2}-\beta}\right)
	\]
	for some sufficiently large $r$ and suitable  $\gamma\in\frac1p\Gamma$,
	$\gamma_1,\dots,\gamma_n,\beta\in\frac1p\Gamma_{>0}$, and $g\in
	k\left[t^{\gamma_1},\dots,t^{\gamma_n}\right][[S_r]]$. By Proposition \ref{truncation} (1), applied to
	$\frac{s_r}{t^\beta}+t^{\frac{p+1}{p^2}-\beta}$ instead of $w+ct^\beta$, we see that there exists $r'\ge r$ such that
	\begin{equation}
v(z)=v(f(x))=v\left(t^\gamma g\left(\frac{s_{r'}}{t^\beta}+t^{\frac{p+1}{p^2}-\beta}\right)\right)\in\frac1p\Gamma=vK',
	\end{equation}
	as desired. This completes the proof that $L|K'$ and $L|K$ are defect extensions, so $K$ is not simply defectless.
	\medskip
		
	To prove that $K$ is algebraically maximal, it is sufficient to prove that $K$ does not admit immediate extensions of degree $p$, that is, Artin--Schreier extensions or purely inseparable degree $p$ extension. Take an element
	\[
	z\in K\setminus K^p,
	\]
	let $y=z^{\frac1p}$ and consider the extension $L=K\left(z^{\frac1p}\right)=K(y)=\frac{K[Y]}{(Y^p-z)}$. Replacing $z$ by $z-a^p$ with $a$ as in Lemma \ref{subtractppowers}, we may assume that $vz\notin p\Gamma$,  which implies that our extension $L$ is not immediate.
	\medskip
	
	Next, let $L=K(x)$ be an Artin-Schreier extension of $K$:
	\[
	L=K(x)=\frac{K[X]}{(X^p-X-b)}
	\]
for some $b\in K$ with $vb<0$ (we may assume that the coefficient of $X$ is 1 because $K$ is closed under taking $(p-1)$-st roots). Write
\begin{equation}
 b=\sum\limits_{(\epsilon_i,j)\in\N^2}a_{\epsilon_ij}t^{\epsilon_i}\left(\frac{w_r}{t^\beta}\right)^j\label{eq:b=}
\end{equation}
as in Corollary \ref{coryEqualsum} and Remark \ref{bounded->finite}. Define the set $\Delta(b)\subset\Gamma\times\N$ as follows:
	\begin{equation}
\Delta(b)=\left\{(\epsilon_i,j)\in\Gamma\times\N\ \left|\
a_{\epsilon_ij}\ne0,\ v\left(t^{\epsilon_i}\left(\frac{w_r}{t^\beta}\right)^j\right)<0,\ p\,|\,(\epsilon_i,j)\right.\right\}.
	\end{equation}
The set $\Delta(b)$ is finite by Remark \ref{bounded->finite}. Put
$$
n(b) = \max\{n\in \N\mid \text{there exists }(\epsilon_i,j)\in\Delta(b)\text{ with }p^n\ |\ (\epsilon_i,j)\}.
$$
In the case when $\Delta(b)=\emptyset$, we adopt the convention $n(b)=0$. Let us prove that $L$ is not immediate by induction on $n(b)$. If $n(b)=0$ then the leading term $t^{\epsilon_i}\left(\frac{w_r}{t^\beta}\right)^j$ in \eqref{eq:b=} has the property $p\ \nmid(\epsilon_i,j)$. Hence  $p\ \nmid\ vb$, so $[vL:\Gamma]=p$ and $L|K$ is not immediate. 
	
Next, assume that $n(b)>0$. Put
$a:=\sum\limits_{(\epsilon_i,j)\in\Delta(b)}a_{\epsilon_ij}^{\frac1p}t^{\frac{\epsilon_i}p}\left(\frac{w_r}{t^\beta}\right)^{\frac jp}$. Let
$x_1=x-a$, so $K(x)=K(x_1)$. Then $x_1$ satisfies the equation $x_1^p-x_1-b+a^p-a=0$.  We claim that
\begin{equation}
\Delta(b-a^p+a)\subset\frac1p\Delta(b).\label{eq:Deltaoverp}
\end{equation}
Indeed, rewrite \eqref{eq:b=} as
\begin{equation}
b=\sum\limits_{(\epsilon_i,j)\in\Delta(b)}a_{\epsilon_ij}t^{\epsilon_i}\left(\frac{w_r}{t^\beta}\right)^j+
\sum\limits_{(\epsilon_i,j)\in(\Gamma\times\N)\setminus\Delta(b)}a_{\epsilon_ij}t^{\epsilon_i}\left(\frac{w_r}{t^\beta}\right)^j
\end{equation}
Then
\begin{eqnarray}
b-a^p+a=a&+&\sum\limits_{(\epsilon_i,j)\in(\Gamma\times\N)\setminus\Delta(b)}a_{\epsilon_ij}t^{\epsilon_i}\left(\frac{w_r}{t^\beta}\right)^j=\nonumber\\
\sum\limits_{(\epsilon_i,j)\in\Delta(b)}a_{\epsilon_ij}^{\frac1p}t^{\frac{\epsilon_i}p}\left(\frac{w_r}{t^\beta}\right)^{\frac jp}&+&
\sum\limits_{(\epsilon_i,j)\in(\Gamma\times\N)\setminus\Delta(b)}a_{\epsilon_ij}t^{\epsilon_i}\left(\frac{w_r}{t^\beta}\right)^j.\label{eq:b1=}
\end{eqnarray}
Now, if $a_{\epsilon_ij}\ne0$ and $(\epsilon_i,j)\in(\Gamma\times\N)\setminus\Delta(b)$ then
$(\epsilon_i,j)\in(\Gamma\times\N)\setminus\Delta(b-a^p+a)$; no term in the second sum on the right hand side of
\eqref{eq:b1=} contributes to $\Delta(b-a^p+a)$. On the other hand, an element of the form
$\left( \frac{\epsilon_i}p,\frac jp\right)$ with $(\epsilon_i,j)\in\Delta(b)$ (appearing in the first sum  on the right hand side of \eqref{eq:b1=}) may or may not belong to $\Delta(b-a^p+a)$; this proves the inclusion \eqref{eq:Deltaoverp}. From \eqref{eq:Deltaoverp} we obtain $n(b-a^p+a)\le n(b)-1$. Thus the extension $L|K$ is not immediate by the induction assumption. This completes the proof of the theorem.
\end{proof}
\medskip


\vspace{0.5cm}

\noindent{\footnotesize CAIO HENRIQUE SILVA DE SOUZA$^1$ (corresponding author)\\
	Federal University of São Carlos, Department of Mathematics\\
	Rodovia Washington Lu\'is, 235\\
	13565-905, S\~ao Carlos - SP, Brasil.\\
	and \\
	Institut de Mathématiques de Toulouse \\
	118, rte de Narbonne, 31062 Toulouse cedex 9, France\\
	Email: {\tt caiohss@estudante.ufscar.br} \\\\

		\noindent{\footnotesize MARK SPIVAKOVSKY$^2$\\
			CNRS UMR 5219 and Institut de Mathématiques de Toulouse \\
			118, rte de Narbonne, 31062 Toulouse cedex 9, France\\
			and\\
			Instituto de Matemáticas (Unidad Cuernavaca) LaSol, UMI CNRS 2001\\
			Universidad Nacional Autónoma de México\\
			Av. Universidad s/n. Col. Lomas de Chamilpa\\
			Código Postal 62210, Cuernavaca, Morelos, México.\\
			Email: {\tt spivakovsky@math.univ-toulouse.fr} \\\\

\end{document}